\documentclass[11pt]{amsart}
\usepackage{amsmath, amssymb, amsthm, amsfonts,enumitem}
\usepackage[colorlinks = true, citecolor = blue]{hyperref}
\usepackage{tikz}
\usetikzlibrary{positioning}

\theoremstyle{plain}
\newtheorem{theorem}{Theorem}[section]
\newtheorem{corollary}[theorem]{Corollary}
\newtheorem{lemma}[theorem]{Lemma}
\newtheorem{proposition}[theorem]{Proposition}
\newtheorem{fact}[theorem]{Fact}
\newcounter{MainCorollaryCounter}
\newtheorem{MainCorollary}[MainCorollaryCounter]{Corollary}

\newcounter{MainTheoremCounter}
\newtheorem{MainTheorem}[MainTheoremCounter]{Theorem}

\theoremstyle{definition}
\newtheorem{definition}[theorem]{Definition}

\newtheorem{remark}[theorem]{Remark}

\newtheorem*{setup}{Setup}
\newcounter{MainProblemCounter}
\newtheorem{MainProblem}[MainProblemCounter]{Problem}

\newcounter{MainPrimeProblemCounter}
\newtheorem{MainPrimeProblem}[MainPrimeProblemCounter]{Problem}

\newcommand{\textdef}[1]{\textit{#1}}

\newcommand{\complex}{\mathbb{C}}

\newcommand{\fieldF}{\mathbb{F}}

\newcommand{\modu}[1]{\overline{#1}}
\newcommand{\orbit}{\mathcal{O}}

\newcommand{\li}[1]{\ell_{#1}}

\DeclareMathOperator{\rk}{rk}

\DeclareMathOperator{\fp}{Fix}

\DeclareMathOperator{\aut}{Aut}

\DeclareMathOperator{\sym}{Sym} 
  
\DeclareMathOperator{\gl}{GL} 

\DeclareMathOperator{\pgl}{PGL} 
\DeclareMathOperator{\psl}{PSL} 
 
\DeclareMathOperator{\agl}{AGL} 
 
\DeclareMathOperator{\proj}{\mathbb{P}} 

\DeclareMathOperator{\points}{\mathcal{P}} 
\DeclareMathOperator{\lines}{\mathcal{L}}

\begin{document}
\title[Recognizing $\pgl_n$ via generic transitivity]{Towards the recognition of $\pgl_n$ via a high degree of generic transitivity}
\author{Tuna Alt\i{}nel}
\address{Universit\'e de Lyon \\
Universit\'e Claude Bernard Lyon 1\\
CNRS UMR 5208\\
Institut Camille Jordan\\
43 blvd du 11 novembre 1918\\
F-69622 Villeurbanne cedex\\
France}
\email{altinel@math.univ-lyon1.fr}
\author{Joshua Wiscons}
\address{Department of Mathematics and Statistics\\
California State University, Sacramento\\
Sacramento, CA 95819, USA}
\email{joshua.wiscons@csus.edu}
\thanks{The second author was partially supported by the Sacramento State Research and Creative Activity Faculty Awards Program.}
\date{\today}
\begin{abstract}
In 2008,  Borovik and Cherlin posed the problem of showing that the degree of \emph{generic} transitivity of an infinite permutation group of finite Morley rank $(X,G)$ is at most $n+2$ where $n$ is the Morley rank of $X$. Moreover, they conjectured that the bound is only achieved (assuming transitivity) by $\pgl_{n+1}(\fieldF)$ acting naturally on projective $n$-space. We solve the problem under the two additional hypotheses that (1) $(X,G)$ is $2$-transitive, and (2) $(X-\{x\},G_x)$ has a definable quotient equivalent to $(\proj^{n-1}(\fieldF),\pgl_{n}(\fieldF))$. The latter hypothesis drives the construction of the underlying projective geometry and is at the heart of an inductive approach to the main problem.
\end{abstract}
\maketitle

\section{Introduction}
Groups of finite Morley rank (fMr) are equipped with a model-theoretic notion of dimension generalizing the usual Zariski dimension for algebraic groups. Affine algebraic groups over algebraically closed fields are the primary examples, and as such, the study of groups of fMr is heavily inspired by the algebraic theory. However, as there is no topology, many of the algebraic tools have rudimentary analogs at best, and the methods of proof often follows lines from finite group theory. The guiding problem for the theory of groups of fMr is the classification of the simple ones, which is motivated by the \emph{Algebraicity Conjecture} of Gregory Cherlin and Boris Zilber: the simple groups of fMr are simple algebraic groups over algebraically closed fields.

The theory of groups of fMr developed over the past forty years with a high point being the resolution of the Algebraicity Conjecture for those groups which contain an infinite elementary abelian $2$-group \cite{ABC08}. Much remains to be addressed around the Algebraicity Conjecture with a key issue being the lack of a Feit-Thompson Theorem---it is unknown if there exist simple groups of fMr without elements of order two. Nevertheless, the field has achieved a sufficient level of maturity that some attention is shifting to general questions about \emph{permutation groups}.

The current research on permutation groups is primarily organized around problems raised by Alexandre Borovik and Gregory Cherlin in \cite{BoCh08}. The main result of \cite{BoCh08} is that there exists a global bound on the rank of a \emph{primitive} permutation group of fMr as a function of the rank of the set being acted upon, and the problems they pose are focused on tightening  this bound. Their proof begins with the nontrivial observation that a bound on the rank of a primitive group can be derived from a bound on the \emph{degree of generic transitivity} of the action (defined below), and it is an investigation of this latter bound that they propose as way to tighten the former one.

\begin{definition}
A permutation group of fMr $(X,G)$ is called \textdef{generically $n$-transitive} if the induced action of $G$ on $X^n$ has an orbit $\orbit$ for which $\rk (X^n - \orbit) < \rk (X^n)$. If $(X,G)$ is generically $n$-transitive but not generically $(n+1)$-transitive, we say that $n$ is the \textdef{degree of generic transitivity}.
\end{definition}

The notion of generic $n$-transitivity is quite natural---much more natural than ordinary $n$-transitivity in our context. The canonical example is that of $\gl_n(\complex)$ acting on the vector space $V:= \complex^n$. The action is quite far from being $2$-transitive (and not even transitive), but it is generically $n$-transitive. The large orbit $\orbit$ of $\gl_n(\complex)$ on $V^n$ is precisely the set of bases for $V$. In \cite{BoCh08}, Borovik and Cherlin pose the following problem, which  seeks to establish a \emph{natural} bound on the degree of generic transitivity.

\begin{MainProblem}[\protect{\cite[Problem~9]{BoCh08}}]\label{prob.ProbPGL}
Show that if $(X,G)$ is an infinite, transitive, generically $(n+2)$-transitive permutation group of fMr with $\rk X = n$, then
$(X,G)\cong(\proj^{n}(\fieldF),\pgl_{n+1}(\fieldF))$ for some algebraically closed field $\fieldF$.
\end{MainProblem}

The early work of Ehud Hrushovski with actions on strongly minimal sets solves Problem~\ref{prob.ProbPGL}  when $X$ has rank $1$ (see \cite[Theorem~11.98]{BoNe94}). The first investigation of the rank $2$ case was by Ursula Gropp in 1992 \cite{GrU92}; the full solution for rank $2$ is much more recent \cite{AlWi15}. 

Generic $n$-transitivity has also been studied in the algebraic category where there is an open question about the degree of generic transitivity for actions of the simple algebraic groups on the various coset spaces of maximal parabolic subgroups. This question was raised by Vladimir Popov in \cite{PoV07} where he also answered it in characteristic~$0$. Popov's work may even have concrete implications for our context (via an analog of the O'Nan-Scott Theorem for primitive groups developed by Dugald Macpherson and Anand Pillay in \cite{MaPi95}); however, there are several issues to overcome with such an approach, not the least of which is the restriction on the characteristic.

In this article, we take up Problem~\ref{prob.ProbPGL} in general, solving it under two additional hypotheses that appear to be the essential ingredients needed to recognize the underlying projective geometry.  To streamline the discussion, we introduce some terminology.

\begin{definition}
An infinite permutation group $(X,G)$ of fMr with $n:= \rk X$ is said to be \textdef{extremal} if it is transitive and generically $(n+2)$-transitive. If $(X,G) \cong (\proj^m(\fieldF),\pgl_{m+1}(\fieldF))$ for some algebraically closed field $\fieldF$ of rank~$r$, then we say that $(X,G)$ is \textdef{projective over a  field of rank $r$}.
\end{definition}

It is not hard to see that projective permutation groups over a field of rank $1$ are extremal---in this new language, Problem~\ref{prob.ProbPGL} seeks the converse. 

\begin{MainPrimeProblem}\label{prob.ProbPrimePGL}
Show that every extremal permutation group of fMr is projective over a  field of rank $1$.
\end{MainPrimeProblem}

To state our main result, we first recall that whenever $(X,G)$ is a permutation group and $\sim$ is a $G$-invariant equivalence relation on $X$, then $G$ acts naturally on $X/{\sim}$, and if $\modu{G}$ is the image of $G$ in $\sym\left(X/{\sim}\right)$, the permutation group $(X/{\sim},\modu{G})$ is called a \textdef{quotient} of $(X,G)$. 

\begin{MainTheorem}\label{thm.A}
Let $(X,G)$ be an extremal permutation group of fMr. Further, assume that 
\begin{enumerate}
\item $(X,G)$ is $2$-transitive, and 
\item for some $x\in X$, $(X-\{x\},G_x)$ has a definable quotient, with classes of infinite size, that is projective over a  field of rank $1$.
\end{enumerate}
Then $(X,G)$ is projective over a  field of rank $1$.
\end{MainTheorem}

The second hypothesis is key, and it is at the heart of an inductive approach to Problem~\ref{prob.ProbPGL}. Indeed, taking such an approach, the existence of a definable quotient of $(X-\{x\},G_x)$ with infinitely many classes of infinite size immediately implies that the quotient is projective over a  field of rank $1$. We formalize this with a corollary. Recall that a permutation group $(X,G)$ is \textdef{virtually (definably) primitive} if every (definable) $G$-invariant equivalence relation on $X$ either has finite classes or finitely many classes.

\begin{MainCorollary}\label{cor.A}
Let $(X,G)$ be an extremal permutation group of fMr, and assume Problem~\ref{prob.ProbPGL} is solved for sets of rank less than $\rk X$. Then 
$(X,G)$ is virtually definably primitive, and we have the following case division:
\begin{enumerate}
\item $(X,G)$ is projective over a  field of rank $1$;
\item\label{cor.Item.2} $(X,G)$ is not $2$-transitive;
\item\label{cor.Item.3}  $(X,G)$ is $2$-transitive with $(X-\{x\},G_x)$ virtually definably primitive for all $x\in X$.
\end{enumerate}
\end{MainCorollary}

The idea behind our proof of Theorem~\ref{thm.A} is that if the elements of $X$ are the points of the expected  geometry, then the classes of the quotient of $(X-\{x\},G_x)$ ought to define the lines through $x$. The $2$-transitivity hypothesis keeps things uniform and allows us to fairly easily define the line through two points. 
Theorem~\ref{thm.A} is a general version of \cite[Proposition~5.1]{AlWi15}, but, notably, the so-called Fixed-point Assumption has now been removed. (However, our condition on the point stabilizer is now slightly stronger.) 

Finally, we should emphasize that our proof of Theorem~\ref{thm.A} is highly geometric and requires very little analysis of the internal structure of $G$. As in \cite{AlWi15}, this will likely stand in sharp contrast to the treatment of the third case in Corollary~\ref{cor.A} (and to some degree the second one as well), which we find to be a good reason for isolating the present theorem.

\section{Background}

We assume familiarity with the theory of groups of fMr, and refer the reader to  \cite{PoB87}, \cite{BoNe94}, and \cite{ABC08} for the necessary background.  Those unfamiliar with Morley rank will do well to translate ``rank'' to ``dimension'' and ``definable'' to ``constructible.'' 

Regarding \emph{permutation groups} of fMr, we provide the relevant definitions as needed and recommend \cite{BoCh08} and \cite{AlWi15} to fill in the details. The solutions to Problem~\ref{prob.ProbPGL} in rank $1$ and $2$ are included below as they, among other things, allow us to ignore small geometries and focus on the generic case.

\begin{fact}[Hrushovski, see \protect{\cite[Theorem~11.98]{BoNe94}}]\label{fact.Hru}
If $(X,G)$ is a transitive permutation group of fMr with $X$ of rank and (Morley) degree $1$, then either
\begin{enumerate}
\item $G^\circ$ is abelian, acting regularly on $X$,
\item $(X,G)\cong (\fieldF,\agl_1(\fieldF))$ for some algebraically closed field $\fieldF$, or
\item $(X,G)\cong (\proj^1(\fieldF),\psl_2(\fieldF))$ for some algebraically closed field $\fieldF$.
\end{enumerate}
\end{fact}

\begin{fact}[\protect{\cite[Theorem~A]{AlWi15}}]\label{fact.rankTwoActions}
If $(X,G)$ is a transitive and generically $4$-transitive permutation group of fMr with $X$ of rank $2$, then
$(X,G)\cong(\proj^{2}(\fieldF),\pgl_{3}(\fieldF))$ for some algebraically closed field $\fieldF$.
\end{fact}

We now end this (very brief) section with some terminology that is frequently used in the sequel.

\begin{definition}
Let $(X,G)$ be a generically $n$-transitive permutation group of fMr. We say that $(x_1,\ldots,x_n) \in X^n$ 
is in \textdef{general position} if the orbit of $G$ on $X^n$ containing $(x_1,\ldots,x_n)$ is of maximal rank. The stabilizers of $n$-tuples in general position will be called \textdef{generic $n$-point stabilizers}.
\end{definition}

\begin{remark}
It is easy to see that if $(x_1,\ldots,x_n)\in X^n$ are in general position, then any permutation of the coordinates is as well (\cite[Lemma~4.8]{AlWi15}). Thus, we often just write $x_1,\ldots,x_n\in X$ are in general position.
\end{remark}

\section{Defining the geometry}\label{sec.DefineGeometry}
From the hypotheses of Theorem~\ref{thm.A}, we define a notion of collinearity for the elements of $X$ with respect to which we obtain a projective space. The goal of this section is simply the definition of collinearity and its connection to the given equivalence relation on $X-\{x\}$. 

By Facts~\ref{fact.Hru} and \ref{fact.rankTwoActions}, Theorem~\ref{thm.A} holds when $X$ has rank $1$ or $2$, so we will assume throughout that $\rk X \ge 3$. We adopt the following setup for the remainder of  this section.

\begin{setup}
Let $(X,G)$ satisfy all hypotheses of Theorem~\ref{thm.A}. Set $n := \rk X$, and assume $n\ge 3$. Fix $1\in X$. Define $X_1 := X-\{1\}$, and let $\sim_1$ be a fixed, definable $G_1$-invariant equivalence relation on $X_1$ such that 
\begin{itemize}
\item $\modu{X_1} := X_1/{\sim_1}$ contains classes of infinite size, and
\item $(\modu{X_1},\modu{G_1})$ is projective over a field of rank $1$, where $\modu{G_1}$ is the image of $G_1$ in $\sym(\modu{X_1})$.  
\end{itemize}
Finally, let $K_1$ be the kernel of the map $G_1 \rightarrow\modu{G_1}$.
\end{setup}

As we are assuming that $(X,G)$ is $2$-transitive, $(X_1,G_1)$ is transitive, so the $\sim_1$-classes have constant rank at least $1$. Thus, $1\le \rk \modu{X_1} \le n-1$. We are also assuming that $(\modu{X_1},\modu{G_1}) \cong (\proj^{m}(\fieldF),\pgl_{m+1}(\fieldF))$ for some algebraically closed field $\fieldF$ of rank $1$, and it is not hard to see that in fact $m = n-1$. Indeed, as $(X,G)$ is generically $(n+2)$-transitive, $(X_1,G_1)$ is generically $(n+1)$-transitive by general principles, so $(\modu{X_1},\modu{G_1})$ is also generically $(n+1)$-transitive, see \cite[Lemma~6.1]{BoCh08} or \cite[Lemma~4.17]{AlWi15}. This forces $m\ge n-1$, so as $m = \rk \modu{X_1} \le n-1$, we have   $m = n-1$. We record this observation in the following remark.

\begin{remark}\label{rem.RankLines}
We have that $(\modu{X_1},\modu{G_1})\cong (\proj^{n-1}(\fieldF),\pgl_{n}(\fieldF))$ for some algebraically closed field $\fieldF$ of rank $1$. Consequently, $\modu{X_1}$ has rank $n-1$, and every $\sim_1$-class has rank~$1$.
\end{remark}

We begin with a lemma describing the action of point-stabilizers in the quotient; it will be followed by a more concrete description of $\sim_1$. However, we first require a definition.

\begin{definition}
A generically $m$-transitive permutation group of fMr is said to be generically \textdef{sharply} $m$-transitive if the stabilizer of any $m$ points in general position is trivial.
\end{definition}

Note that $(\proj^{n-1}(\fieldF),\pgl_{n}(\fieldF))$ is generically sharply $(n+1)$-transitive. Also, if $(Y,H)$ is any generically sharply $m$-transitive permutation group of fMr, then there exists a definable bijection between $H$ and the large orbit $\orbit\subset Y^m$, implying that $\rk H = m\rk Y$.

Regarding notation, if $x\in X$, then $\overline{x}$ is the image of $x$ in $\modu{X_1}$, and by $G_{\overline{x}}$, we mean the setwise stabilizer of $\overline{x}$ in $G$. Then, $G_{1,\overline{x}} = G_1 \cap G_{\overline{x}}$, so $\modu{G_{1,\overline{x}}} = (\modu{G_1})_{\overline{x}}$, where the latter is the stabilizer in $\modu{G_1}$ of the ``point'' $\overline{x}$.

\begin{lemma}\label{lem.ImagePointStabilizersInQuotient}
If $2\le k\le n+1$ and $1,2,\ldots,k\in X$ are in general position, then $\modu{G_{1,2,\ldots,k}} = \modu{G_{1,\overline{2},\ldots,\overline{k}}}$.
\end{lemma}
\begin{proof}
Let $A:= G_{1,2,\ldots,k}$, and $B := G_{1,\overline{2},\ldots,\overline{k}}$. Since $(X,G)$ is generically $(n+2)$-transitive, we find that $A$ acts generically $(n+2-k)$-transitively on $X$, so $\modu{A}$ acts generically $(n+2-k)$-transitively on $\modu{X_1}$ as well (see \cite[Lemma~6.1]{BoCh08} or \cite[Lemma~4.17]{AlWi15}). Now, we know the action of $\modu{B}$ on $\modu{X_1}$; it is acting as the stabilizer in $\pgl_{n}$ of $(k-1)$ points of projective $(n-1)$-space. Thus, $\modu{B}$ acts generically \emph{sharply} $(n+2-k)$-transitively on $\modu{X_1}$. Since $\modu{A}\le \modu{B}$, we find that $\modu{A}=\modu{B}$.
\end{proof}

\begin{lemma}\label{lem:UniqueSim}
Let $2\in X_1$. Then every $G_{1,2}$-orbit on $X$ of rank at most $1$ belongs to $\overline{2}\cup\{1\}$.
\end{lemma}
\begin{proof}
Set $A:=\overline{2}\cup\{1\}$. Since $A$ is $G_{1,2}$-invariant, the fact that $A$ has rank $1$ (by Remark~\ref{rem.RankLines}) implies that every element of $A$ is contained in a $G_{1,2}$-orbit of rank at most $1$. Now assume  $x\notin A$. We want to show that the orbit of $G_{1,2}$ on $x$ has rank at least $2$. By $2$-transitivity of $(X,G)$, $1$ and $2$ are in general position. By Lemma~\ref{lem.ImagePointStabilizersInQuotient},  $\modu{G_{1,2}} = \modu{G_{1,\overline{2}}}$, so as $(\modu{X_1},\modu{G_1})$ is $2$-transitive, the orbit of $\modu{G_{1,2}}$ on $\modu{x}$ has rank $n-1$. Since we are assuming that $n\ge 3$, the orbit of $\modu{G_{1,2}}$ on $\modu{x}$ is at least $2$, so the same must be true of $G_{1,2}$ on $x$.
\end{proof}

We now define the lines of our geometry. The connection with $\sim_1$ will be made explicit below in Proposition~\ref{prop.LinesAndSim}.

\begin{definition}
Define a relation $\lambda$ on $X^3$ by $\lambda(a,b,c)$ if and only if the orbit of $G$ on the triple $(a,b,c)$ has rank at most $2n+1$. 
\begin{itemize}
\item We read $\lambda(a,b,c)$ as ``$a$, $b$, $c$  are \textdef{collinear}.'' 
\item If $a$, $b$, $c$  are not collinear, we say they form a \textdef{triangle}.
\item For $a\neq b\in X$, define $\li{ab}:= \lambda(a,b,X)$.
\end{itemize}
A \textdef{line} will be any subset of $X$ of the form $\li{ab}$ with $a\neq b\in X$. The set of lines is denoted $\lines$, and the set of lines through $a$ is  $\lines(a) := \{\ell_{ab}: b\in X -\{a\}\}$.
\end{definition}

\begin{remark}\hfill
\begin{enumerate}
\item The orbits of $G$ on $X^3$ are uniformly definable, as are their ranks, so $\lambda$ is definable.
\item Note that triples with repeated entries are always collinear since the orbit of $G$ on such a triple will have rank at most $2n$. Thus, triples forming a triangle must not have repeated entries.
\item  As $(X,G)$ is generically $3$-transitive, any three points in general position form a triangle, but, a priori, there is no reason to believe that the points forming a triangle are necessarily in general position. That is, at this point, there may be different types of triangles; however, Lemma~\ref{lem.TrianglesAreInGP} will rule this out.
\item Note that $\li{ab} = \li{ba}$.
\end{enumerate}
\end{remark}

\begin{proposition}\label{prop.LinesAndSim}
If $2\in X_1$, then $\li{12} = \overline{2}\cup\{1\}$; thus $(\lines(1),\modu{G_1}) \cong (\modu{X_1},\modu{G_1})$.
\end{proposition}
\begin{proof}
By $2$-transitivity of $(X,G)$, the orbit of $G$ on the pair $(1,2)$ has rank $2n$. Thus, the points collinear with $1$ and $2$ are precisely those $c$ for which the orbit $cG_{1,2}$ has rank at most $1$, so by Lemma~\ref{lem:UniqueSim}, $\li{12} = \overline{2}\cup\{1\}$.
\end{proof}

Of course, the element $1\in X$ is arbitrary; this leads to the following corollary, which also incorporates Remark~\ref{rem.RankLines}.

\begin{corollary}\label{cor.UniqueEqRel}
For all $x\in X$, the relation on $X_x := X-\{x\}$ given by $y\sim z$ if and only if $\li{xy} = \li{xz}$ is an equivalence relation, and  if $K$ is the kernel of the action of $G_x$ on $\lines(x)$, then $(\lines(x),G_x/K) \cong  (\proj^{n-1}(\fieldF),\pgl_{n}(\fieldF))$.
\end{corollary}

\section{Analyzing the geometry}
We now begin to work out the details of our geometry. We start by recalling the axioms for a projective space. 

\begin{definition}
A geometry consisting of points, lines, and an incidence relation is called a \textdef{projective space} provided
\begin{enumerate}
\item\label{def.PS.axiom1} any two distinct points line on a unique line;
\item (Veblen's Axiom) if $a,b,c,d$ are distinct points and the line through $a$ and $b$ meets the line through $c$ and $d$, then the line through $b$ and $c$ meets the line through $a$ and $d$; and
\item there exist four points, no three of which are collinear.
\end{enumerate}
\end{definition}

\begin{setup}
Let $(X,G)$ satisfy the hypotheses of Theorem~\ref{thm.A} with $n := \rk X$. Set $\points := X$ and $\lines := \{\li{ab}: a\neq b\in X\}$ (as defined in the previous section). 
\end{setup}

The goals of this section are to show that two distinct points lie on a unique line (Axiom~\ref{def.PS.axiom1} above) and to prepare for the proof of Veblen's Axiom. 

\begin{proposition}[Two points determine a unique line]
Let $a,b\in \points$ be distinct. Then $\li{ab}$ is the \emph{unique} line containing $a$ and $b$. Hence, any two lines intersect in at most $1$ point. 
\end{proposition}
\begin{proof}
First, let $u,v,w\in \points$ with $u\neq v$ and $u\neq w$. Since, by Corollary~\ref{cor.UniqueEqRel} the sets $\li{uv} - \{u\}$ and $\li{uw} - \{u\}$ are classes of an equivalence relation, we have that if $w\in \li{uv}$ and $w\neq u$, then $\li{uv} = \li{uw}$. 

Now assume that $a,b \in \li{pq}$ with $p$ and $q$ distinct elements of $\points$. We may assume that $a\neq p$, so by the previous observation, $\li{pq} = \li{pa} = \li{ap}$. Now $b\in \li{ap}$ (since $\li{ap} = \li{pq}$) and $b\neq a$, so again using the above observation, $\li{ap} = \li{ab}$. Thus,  $\li{pq} = \li{ab}$.
\end{proof}

We now begin to investigate the structure of $G$. Perhaps surprisingly, we already have enough information to compute its rank.  

\begin{corollary}\label{cor.RankG}
The action of $G$ on $X$ is generically sharply $(n+2)$-transitive, and consequently, $G$ is connected of rank $n(n+2)$. Moreover, all generic $k$-point-stabilizers are connected for $k\le n+2$.
\end{corollary}
\begin{proof}
Let $1,\ldots,n+2 \in \points$ be in general position.  Let $K_1$ be the kernel of the action of $G_1$ on $\lines(1)$, and similarly define $K_2$. (So $K_1$ fixes each line through $1$ setwise, and similarly for $K_2$.) The key observation is that any point $z\in\points-\{1,2\}$ is the \emph{unique} point of intersection of $\li{1z}$ and $\li{2z}$, so $z$ is fixed by $K_1\cap K_2$. As the action of $G$ on $\points$ is faithful, $K_1\cap K_2 = \{1\}$.

By \cite[Lemma~4.17]{AlWi15}, $(\ell_{12},\ell_{13},\ldots,\ell_{1(n+2)})$ is a generic tuple in $\lines(1)^{n+1}$, so as $(\lines(1),\modu{G_1}) \cong (\proj^{n-1}(\fieldF),\pgl_{n}(\fieldF))$, the stabilizer in $\modu{G_1}$ of the tuple is trivial. Of course, $G_{1,2,\ldots,n+2}$ fixes each of the lines $\ell_{12},\ell_{13},\ldots,\ell_{1(n+2)}$, so $G_{1,2,\ldots,n+2} \le K_1$. Similarly, we find that $G_{1,2,\ldots,n+2} \le K_2$, so $G_{1,2,\ldots,n+2} = 1$. 

We now have that the action of $G$ on $X$ is generically sharply $(n+2)$-transitive, so there is a definable \emph{bijection} between $G$ and the orbit  $\orbit := (1,\ldots,n+2 )\cdot G$ in $X^n$. This shows that $\rk G = n(n+2)$. Also, since $(X,G)$ is generically $2$-transitive, $G^\circ$ has a single orbit on $X$, so $X$ has degree $1$. This implies that $X^n$, hence $\orbit$, also has degree $1$, so as $G$ is in definable bijection with $\orbit$, $G$ is connected. The connectedness of the generic $k$-point-stabilizers follows by a similar argument.
\end{proof}

The next lemma, in conjunction with Fact~\ref{fact.Hru}, says that the group induced by $G$ on any line is of the form $\agl_1$ or $\psl_2$. 

\begin{lemma}
If $\ell\in \lines$, then $G_\ell$ acts $2$-transitively on $\ell$, and consequently, $\ell$ has degree $1$.
\end{lemma}
\begin{proof}
Set $N:= G_\ell$. Fix $x\neq y\in \ell$. Let $z\in \ell= \li{xy}$ with $z$ different from $x$ and $y$. Since $G_x$ is transitive on $\points-\{x\}$, there is a $g\in G_x$ taking $y$ to $z$. Of course, this implies that $g \in N$, so $N_x$ is transitive on $\ell - \{x\}$. Similarly, we find that $N_y$ is transitive on $\ell - \{y\}$, so $N$ is $2$-transitive on $\ell$. This also implies that $N^\circ$ must have a single orbit on $\ell$, so $\ell$ has degree $1$.
\end{proof}

We are now in a position to identify the torus. Recall that a \textdef{good torus} of $G$ is a definable divisible abelian subgroup $T$ such that every definable subgroup of $T$ is equal to the definable closure of its torsion subgroup. We will make use of the following fundament fact about  actions of tori.

\begin{fact}[{\cite[Lemma~3.11]{BoCh08}}]\label{fact.ActionsTori}
Let $(Y,H)$ be a transitive permutation group of fMr, $T$ a definable divisible abelian subgroup of $H$, and $O(T)$ the largest definable torsion free subgroup of $T$. Then $\rk(T /O(T))\le \rk Y$.
\end{fact}

\begin{proposition}\label{prop.MaximalTorus}
If $1,\ldots,n+2 \in \points$ are in general position, then $G_{1,\ldots,n+1}$ is a self-centralizing, maximal good torus of $G$. 
\end{proposition}
\begin{proof}
Set $H:= G_{1,\ldots,n+1}$; let $K$ be the kernel of $(\lines(1),G_1)$. By Lemma~\ref{lem.ImagePointStabilizersInQuotient} and the fact that $(\lines(1),G_1)\cong (\proj^{n-1}(\fieldF),\pgl_{n}(\fieldF))$, $\modu{H}$, hence $H/(H\cap K)$ is a good torus of rank $n-1$. As a consequence of Corollary~\ref{cor.RankG}, $K$ has rank $ n+1$, and $A:= (H\cap K)^\circ$ has rank $1$.  We must show that $A$ is a good torus.

Set $\ell_i := \ell_{1i}$, the line through $1$ and $i$, and consider the canonical map from $K$ to $G^{\ell_2}\times\cdots\times G^{\ell_{n+2}}$ where $G^{\ell_{i}}$ is the group induced by $G$ on $\ell_i$. The kernel of the map, $N$, fixes $1,\ldots,n+2$, so by Corollary~\ref{cor.RankG}, $N=1$. Since $(\lines(1),\modu{G_1}) \cong (\proj^{n-1}(\fieldF),\pgl_{n}(\fieldF))$, $G_1$ permutes the lines $\ell_2,\ldots,\ell_{n+2}$ transitively while normalizing $K$, so the images of $K$ in each $G^{\ell_i}$ are isomorphic. 

By the previous lemma,  $G^{\ell_i}$ is of the form $\agl_1$ or $\psl_2$, so as the image of $K$ in each $G^{\ell_i}$ fixes a point (namely $1$), the image is either a good torus or a Borel subgroup of $\psl_2$ of the form $\agl_1$. In the former case, $K$ embeds into a product of good tori. However, this implies that $K^\circ$ is a good torus of rank $n+1$, which contradicts Fact~\ref{fact.ActionsTori}. Thus, the image of $K$ in each $G^{\ell_i}$ is of the form $\agl_1$, so $K^\circ$ is solvable and nonnilpotent. In particular, if $F:=F^\circ(K)$ is the connected Fitting subgroup of $K$, then the image of $F$ in $G^{\ell_i}$ is the unipotent radical of $\agl_1$.

We now claim that $A \nleq F$. Assume $A\le F$. For $2\le i\le n+2$, define \[A_{\hat{i}} := (G_{1,\ldots,\hat{i},\ldots,n+2} \cap K)^\circ\] where $\hat{i}$ denotes that $i$ has been removed; thus, $A = A_{\widehat{n+2}}$. The $A_{\hat{i}}$ are $G_1$-conjugate, so our assumption implies that $A_{\hat{i}} \le F$ for all $2\le i\le n+2$. Now consider the sequence of points stabilizers $F_1, F_{1,2},\ldots, F_{1,\ldots,n+2}$; of course $F = F_1$. Note that, for $2\le i\le n+2$,  $A_{\hat{i}}  \subseteq F_{1,\ldots,i-1}$ with $A_{\hat{i}} \nsubseteq  (F_{1,\ldots,i})^\circ$; this implies that $F$ has rank at least $n+1$. However, we have observed that $K^\circ$ is nonnilpotent of rank $n+1$, so it must be that $A \nleq F$.

Now,  $A$ must have a nontrivial image in some $G^\ell_i$. The image of $F$ in $G^{\ell_i}$ is the unipotent radical of $\agl_1$, so as $A \nleq F$, we find that $A$ is a good torus. We conclude that $H$ is a good torus, which by Fact~\ref{fact.ActionsTori}, is maximal. 
\end{proof}

The previous proposition has an extremely important consequence for the action of $G$, namely that the so-called Fixed-point Assumption holds (cf. \cite[Proposition~5.1(3)]{AlWi15}). The proof utilizes (very lightly) the $\Sigma$-groups from \cite[Section~4.4]{AlWi15}, which we now define.

\begin{definition}
 If $x_1,\ldots,x_m\in \points$ are in general position,  we define the group $\Sigma(x_1,\ldots,x_{k-1};x_k,\ldots,x_m)$ to be the subgroup of $G$ that normalizes the set $\{x_1,\ldots,x_{k-1}\}$ while fixing $x_k,\ldots,x_m$ pointwise. 
\end{definition}

\begin{corollary}[Fixed-Point Assumption]\label{cor.FPA}
If $1,\ldots,n+1 \in \points$ are in general position, then $\fp(G_{1,\ldots,n+1}) = \{1,\ldots,n+1\}$.
\end{corollary}
\begin{proof}
Set $H:= G_{1,2,\ldots,n+1}$. Let $\ell_i := \ell_{1i}$, and set  $\widehat{H}:=G_{1,\ell_2,\ldots,\ell_{n+1}}$. Assume that $H$ fixes a point $x\notin \{1,2,\ldots,n+1\}$. Then $H$ certainly fixes the line $\ell_{1,x}$ (setwise), so by Lemma~\ref{lem.ImagePointStabilizersInQuotient}, $\widehat{H}$ does as well. However, since $(\lines(1),G_1)\cong (\proj^{n-1}(\fieldF),\pgl_{n}(\fieldF))$, the image of $\widehat{H}$ in $(\lines(1),G_1)$ fixes no more lines than those in $\{\ell_2,\ldots,\ell_{n+1}\}$, so without loss of generality, we may assume that $\ell_{1,x} = \ell_{1,2} = \ell_2$.

Now, the image of $H$ in $G^{\ell_2}$ fixes the points $1$, $2$, and $x$, so by Fact~\ref{fact.Hru}, $H$ acts trivially on $\ell_2$. Moreover, the normalizer of $H$ in $G_1$ contains $\Sigma(2,3,\ldots,n+1;1)$, which permutes the set ${\ell_2,\ldots,\ell_{n+1}}$ transitively, so $H$ acts trivially on $\ell_i$ for every $2\le i\le n+1$. This implies that $H$ coincides with the subgroup of $G_1$ centralizing (the points of) each of the lines $\ell_2,\ldots,\ell_{n+1}$. By definition, the subgroup of $G_1$ fixing each of the lines (setwise) is $\widehat{H}$, so we find that $\widehat{H}$ normalizes $H$. However, by Proposition~\ref{prop.MaximalTorus}, $H$ is a self-centralizing good torus, which implies that $H$ is almost self-normalizing (see \cite[Lemma~4.23]{ABC08}). Thus, $\rk H = \rk \widehat{H}$. But, it is easily seen (for example using Corollary~\ref{cor.RankG}) that $\rk (\widehat{H}/H) = n$,  a contradiction.
\end{proof}

\section{Veblen's Axiom}
We now prove Veblen's axiom, which we restate here for convenience. The claim is that if $x,y,z,w$ are distinct points for which $\li{xy}$ and $\li{zw}$ intersect, then the lines  $\li{yz}$ and $\li{xw}$ intersect as well, see below for a couple of different pictures. Our idea is to show that for some $p\in \li{yz}$ the group $G_{x,y,z}$ contains an element $g$ taking $\li{xp}$ to $\li{xw}$; in this case, $p\cdot g$ will be the desired point of intersection. 
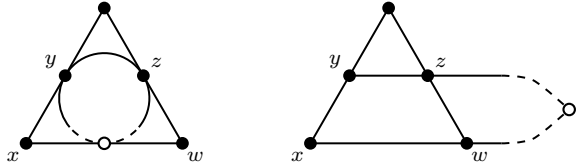
\begin{center}
\begin{tikzpicture}[scale = 1.2,line width = .75]
\newcommand{\dotsize}{2pt}
{\scriptsize

\node (x) at ({cos(210)},{sin(210}) {}; \fill (x) circle (\dotsize);
\node (a) at ({cos(90)},{sin(90}) {}; \fill (a) circle (\dotsize);
\node (w) at ({cos(-30)},{sin(-30}) {}; \fill (w) circle (\dotsize);

\node (y) at ({0.5*cos(150)},{0.5*sin(150}) {}; \fill (y) circle (\dotsize);
\node (z) at ({0.5*cos(30)},{0.5*sin(30}) {}; \fill (z) circle (\dotsize);
\node (q) at ({0.5*cos(270)},{0.5*sin(270}) {}; 

\draw  ({0.5*cos(-40)},{0.5*sin(-40)}) arc (-40:220:{1/2});
\draw[dashed]  ({0.5*cos(225)},{0.5*sin(225}) arc (225:315:{1/2});
\draw plot  coordinates {(x) (a) (w) (x)};

\fill (q) circle ({0.9*\dotsize});
\filldraw[fill = white,line width = .8] (q) circle ({0.9*\dotsize});

\node[anchor = 45] at (x) {$x$}; 
\node[anchor = -45] at  (y) {$y$};
\node[anchor = -135]  at (z) {$z$};
\node[anchor = 135] at (w) {$w$};
}
\end{tikzpicture}
\hspace{20pt}
\begin{tikzpicture}[scale = 1.2,line width = .75]
\newcommand{\dotsize}{2pt}
{\scriptsize

\node (x) at ({cos(210)},{sin(210}) {}; \fill (x) circle (\dotsize);
\node (a) at ({cos(90)},{sin(90}) {}; \fill (a) circle (\dotsize);
\node (w) at ({cos(-30)},{sin(-30}) {}; \fill (w) circle (\dotsize);

\node (y) at ({0.5*cos(150)},{0.5*sin(150}) {}; \fill (y) circle (\dotsize);
\node (z) at ({0.5*cos(30)},{0.5*sin(30}) {}; \fill (z) circle (\dotsize);
\node (q) at (2,-.125) {}; 

\draw plot  coordinates {(x) (a) (w) (x)};
\draw[dashed] plot [smooth] coordinates {(1.25,{0.5*sin(30}) (1.6,{0.5*sin(30) -0.05}) (2,-.125)};
\draw[dashed] plot [smooth] coordinates {(1.25,{sin(-30}) (1.6,{sin(-30) + 0.05}) (2,-.125)};
\draw plot [smooth] coordinates {(y) (z) (1.25,{0.5*sin(30})};
\draw plot  coordinates {(w) (1.25,{sin(-30})};

\fill (q) circle ({0.9*\dotsize});
\filldraw[fill = white,line width = .8] (q) circle ({0.9*\dotsize});

\node[anchor = 45] at (x) {$x$}; 
\node[anchor = -45] at  (y) {$y$};
\node[anchor = -135]  at (z) {$z$};
\node[anchor = 135] at (w) {$w$};
}
\end{tikzpicture}
\end{center}

\begin{setup}
As in the previous section, $(X,G)$ satisfies the hypotheses of Theorem~\ref{thm.A} with $n := \rk X$, $\points := X$, and $\lines := \{\li{ab}: a\neq b\in X\}$. 
\end{setup}

\begin{lemma}[Triangles are in general position]\label{lem.TrianglesAreInGP}
The points $x, y, z$ form a triangle if and only if they are in general position.
\end{lemma}
\begin{proof}
The reverse direction is clear. Now, suppose that $x, y, z$ are not in general position; we show that they are collinear. Assume not, so $\li{xy} \neq \li{xz}$.
Choose $a$ such that $x, y, a$ are in general position. Combining Lemma~\ref{lem.ImagePointStabilizersInQuotient} with the fact that $(\lines(x),G_x)$ is $2$-transitive, we see that $G_{x,y}$ is transitive on $\lines(x) - \li{xy}$, so there exists $g\in G_{x,y}$ taking $\li{xa}$ to $\li{xz}$. Thus, for $w = ag$, we have that $x,y,w$ are in general position, and $\li{xw} = \li{xz}$. We now work to show that $G_{x,y,w}$ fixes $z$, contradicting Corollary~\ref{cor.FPA}.

Since $x, y, z$ are not in general position, the corank of $G_{x,y,z}$ in $G_{x,y}$ is at most $n-1$. Also, by  Lemma~\ref{lem.ImagePointStabilizersInQuotient}, the orbit of $G_{x,y}$ on $\li{xz}$ is generic in $\lines(x)$, so the corank of $G_{x,y,\li{xz}}$ in $G_{x,y}$ is equal to $n-1$. Thus, it must be that $(G_{x,y,z})^\circ = (G_{x,y,\li{xz}})^\circ$. Now, as $\li{xw} = \li{xz}$, we find that $(G_{x,y,w})^\circ \le (G_{x,y,\li{xz}})^\circ  = (G_{x,y,z})^\circ$, so  $(G_{x,y,w})^\circ$, which is equal to $G_{x,y,w}$ by Corollary~\ref{cor.RankG}, fixes $z$. This contradicts Corollary~\ref{cor.FPA}.
\end{proof}

\begin{proposition}
Veblen's axiom holds.
\end{proposition}
\begin{proof}
Let  $x,y,z,w$ be distinct points for which $\li{xy}$ and $\li{zw}$ intersect. We may assume that $\li{xy}\neq\li{zw}$, so $x,y,z$ are in general position (by Lemma~\ref{lem.TrianglesAreInGP}). Set $H:= G_{x,y,z}$. We know that $(\lines(x),\modu{G_x})\cong (\proj^{n-1}(\fieldF),\pgl_{n}(\fieldF))$ where $\modu{G_x} = G_x/K$, with $K$ the kernel of the action of $G_x$ on $\lines(x)$. Additionally, by Lemma~\ref{lem.ImagePointStabilizersInQuotient}, we know that $\modu{H} = \modu{G_{x,\li{xy},\li{xz}}}$. 

The key observation to make is that, by inspection of the action of $\pgl_{n}(\fieldF)$ on $\proj^{n-1}(\fieldF)$, the stabilizer of two lines $\modu{G_{x,\li{xy},\li{xz}}}$  has precisely $4$ orbits on $\lines(x)$: $\{\li{xy}\}$, $\{\li{xz}\}$, one of rank $1$ that we call $Y$, and one of (full) rank $n-1$. And importantly, the same statement holds for $H$ (by Lemma~\ref{lem.ImagePointStabilizersInQuotient}). As such, our goal is to show that  for some $p\in \li{yz}$ with $p$ different from $y$ and $z$, $\li{xp}$ and $\li{xw}$ are both in $Y$. This will imply that there exists an $h\in H$ taking $\li{xp}$ to $\li{xw}$, and as $h$ fixes $\li{yz}$ setwise, $\li{yz}$ and $\li{xw}$ must intersect at the point $p\cdot h$.

Let $p\in\points(\li{yz}) - \{y,z\}$ be arbitrary. Since $H$ fixes $\li{yz}$ setwise, the orbit $pH$ is confined to the rank $1$ set of points on $\li{yz}$. Since each line in the orbit $\li{xp}H$, is determined by ($x$ and) its point of intersection with $\li{yz}$, the orbit $\li{xp}H$ has rank one; thus $\li{xp}\in Y$. 

It remains to show that $\li{xw}\in Y$. Using the Veblen configuration, we first show that $wH$ has rank at most $2$ by identifying an $H$-invariant rank $2$ set containing $w$. Define $A:= \{q:\li{qz}\cap\li{xy} \neq \emptyset\}$. The map $A\rightarrow \points(\li{xy}):q\mapsto p\in \li{qz}\cap\li{xy}$ is a definable surjection with fibers of rank $1$. Thus, $\rk A = 2$. As $w\in A$ and $A$ is $H$-invariant,  $\rk(wH) \le 2$.

Now, towards a contradiction, assume that $\li{xw}\notin Y$, so $\rk (\li{xw} H) = n-1$. Since $\rk (\li{xw} H) \le \rk(wH)$, we have $n-1 = \rk (\li{xw} H) \le \rk(wH) \le 2$, so as $n\ge 3$, it must be that $\rk(\li{xw}H) = \rk(wH) = 2$ and $n=3$. 

Next, we find an $r\in \li{xw}$ such that $x,y,z,r$ are in general position. Indeed, choose $s$ so that $x,y,z,s$ are in general position. Thus, $\li{xy}, \li{xz}, \li{xs}$
are in general position in $\lines(x)$, and by our assumption that $\li{xw}\notin Y$, $\li{xy}, \li{xz}, \li{xw}$ are in general position as well. Then $\li{xw}$ and $\li{xs}$ are both in the generic orbit of $\modu{H} = \modu{G_{x,\li{xy},\li{xz}}}$ on $\lines(x)$, so $\li{xw}$ contains an $H$-conjugate of $s$, which we take to be $r$.

Finally, we obtain the contradiction. We know $\rk(wH)  = \rk\left(\li{xw}H\right) = 2$, so as $H_w \le H_{\li{xw}}$, we find that $\left(H_w\right)^\circ = \left(H_{\li{xw}}\right)^\circ$. Also, $\li{xw} = \li{xr}$, so $H_r \le H_{\li{xw}}$. Now, $H_r$ is connected by Corollary~\ref{cor.RankG}, so $H_r \le \left(H_{\li{xw}}\right)^\circ = \left(H_w\right)^\circ$, which contradicts Corollary~\ref{cor.FPA}. Thus, $\li{xw}\in Y$, and we are done.
\end{proof}

\begin{corollary}
The geometry is that of projective $n$-space for $n\ge 3$.
\end{corollary}
\begin{proof}
It only remains to verify that there exist four points, no three of which are collinear. This is clearly satisfied by four points in general position.
\end{proof}

\section{Proof of the main results}

\begin{proof}[Proof of Theorem~\ref{thm.A}]
Let $(X,G)$ satisfy the hypotheses of Theorem~\ref{thm.A} with $n := \rk X$. By our work above, the geometry with respect to $\points := X$ and $\lines := \{\li{ab}: a\neq b\in X\}$ is that of a projective $n$-space for $n\ge 3$. Our geometry is definable, and as it is automatically arguesian (because $n\ge 3$), it can be definably coordinatized (see for example \cite{ArE57,HiD59}). As our lines have rank $1$, we find that the geometry is that of $\proj^n(\fieldF)$, for some algebraically closed field $\fieldF$ of rank $1$. 

We now have that $G \le \pgl_{n+1}(\fieldF) \rtimes \aut(\fieldF)$. This copy of $\pgl_{n+1}(\fieldF)$ is definable, and we consider $N:= G\cap \pgl_{n+1}(\fieldF)$. Since $G/N$ is a definable group of field automorphisms,   $G/N = 1$ (see \cite[Lemma~4.5]{ABC08}). Thus, $G \le \pgl_{n+1}(\fieldF)$. Further, $\pgl_{n+1}(\fieldF)$ acts generically \emph{sharply} $(n+2)$-transitively on $\points$, so as $G$ acts generically $(n+2)$-transitively, we find that $G=\pgl_{n+1}(\fieldF)$.
\end{proof}

\begin{proof}[Proof of Corollary~\ref{cor.A}]
Let $(X,G)$ be an extremal permutation group of fMr with $n:=\rk X$. Assume Problem~\ref{prob.ProbPGL} is solved for sets of rank less than $n$. 

We first show that the action is virtually definably primitive. If not, there is a definable quotient $(\modu{X},\modu{G})$ with infinitely many classes of infinite size. Thus, $m:= \rk \modu{X} < \rk X$ and $m>0$. Moreover, as we have noted before, $(\modu{X},\modu{G})$  is transitive and generically $(n+2)$-transitive by \cite[Lemma~6.1]{BoCh08} (or \cite[Lemma~4.17]{AlWi15}), so $(\modu{X},\modu{G})$ is transitive and generically $(m+3)$-transitive on a set of rank $m<n$. This contradicts our assumption that Problem~\ref{prob.ProbPGL} is solved for sets of rank less than $n$. 

If $(X,G)$ is not $2$-transitive, we are done (falling into case~\eqref{cor.Item.2} of the statement of the corollary), so assume it is. Then, if $(X-\{x\},G_x)$ is virtually definably primitive for all $x\in X$, we are also done (and in case~\eqref{cor.Item.3}). Thus, $(X,G)$ is $2$-transitive and $(X-\{x\},G_x)$ is virtually definably imprimitive for some $x\in X$. Set $X_x := X-\{x\}$. We now have that $(X_x,G_x)$ is transitive, and by general principles, $(X_x,G_x)$ is generically $(n+1)$-transitive. Let $(\modu{X_x},\modu{G_x})$ be any definable quotient with infinitely many classes of infinite size, and set $r:= \rk \modu{X_x}$. Applying  our inductive hypothesis to $(\modu{X_x},\modu{G_x})$, we find that $(\modu{X_x},\modu{G_x})$ is projective, and we are finished by Theorem~\ref{thm.A}.
\end{proof}

\section{Acknowledgements}
The second author would like to acknowledge the warm hospitality of Universit\'e Claude Bernard Lyon-1, where the majority of the work for this article was carried out, as well as support from Hamilton College and California State University, Sacramento (through the Research and Creative Activity Faculty Awards Program) for making the visits to Lyon possible. The authors would also like to warmly thank the anonymous referee for several helpful comments that served to improve the clarity of the article.
\bibliographystyle{alpha}
\bibliography{WisconsBib}
\end{document}